\documentclass{article}

\usepackage{amsthm,amsmath}
\usepackage[colorlinks,citecolor=blue,urlcolor=blue]{hyperref}

\usepackage{amsmath,amssymb}
\usepackage{amsfonts}
\usepackage{verbatim}
\usepackage{longtable}
\usepackage{lscape}
\usepackage{graphicx}
\usepackage{inputenc}
\usepackage{hyperref}
\usepackage{authblk}

\numberwithin{equation}{section}

\usepackage[english]{babel}

\usepackage[T1]{fontenc}
\usepackage{amscd}

\usepackage{color}

\theoremstyle{plain}
\newtheorem{theorem}{Theorem}[section]

\newtheorem{proposition}[theorem]{Proposition}

\theoremstyle{definition}
\newtheorem{definition}[theorem]{Definition}
\newtheorem{rem}[theorem]{Remark}
\newtheorem{cor}[theorem]{Corollary}

\DeclareMathAlphabet{\mathbbo}{U}{bbold}{m}{n}

\newcommand{\bbR}{\mathbb{R}}

\newcommand{\bbS}{\mathbb{S}}
\newcommand{\mF}{\mathcal{F}}

\newcommand{\mcP}{\mathcal{P}}

\newcommand{\Rn}{\bbR^d}
\newcommand{\bRb}{\bbR}
\newcommand{\Sn}{\bbS^{d-1}}

\newcommand{\rmd}{\mathrm{d}}

\newcommand{\E}{\mathbb{E}}
\newcommand{\R}{\mathbb{R}}

\title{On the exit-problem for self-interacting diffusions}

\author[1]{Ashot Aleksian}
\author[2]{Pierre Del Moral}
\author[3]{Aline Kurtzmann}
\author[1]{Julian~Tugaut}
\affil[1]{Universit\'e Jean Monnet, Institut Camille Jordan, 23, rue du docteur Paul Michelon,
CS 82301,
42023 Saint-\'Etienne Cedex 2,
France.}
\affil[2]{INRIA Bordeaux Research Center, France.}
\affil[3]{Universit\'e de Lorraine, Institut Elie Cartan de Lorraine, CNRS, Institut Elie Cartan de Lorraine, UMR 7502, Vandoeuvre-l\`es-Nancy, F-54506, France. }

\begin{document}

\maketitle

\begin{abstract}
We study the exit-time from a domain of a self-interacting diffusion, where the Brownian motion is replaced by $\sigma B_t$ for a constant $\sigma$. The first part of this work consists in showing that the rate of convergence (of the occupation measure of the self-interacting process toward some explicit Gibbs measure) previously obtained in~\cite{kk-ejp} for a convex confinment potential $V$ and a convex interaction potential can be bounded uniformly with respect to $\sigma$. Then, we prove  an Arrhenius-type law for the first exit-time from a domain (satisfying classical hypotheses of Freidlin-Wentzell theory). 



\end{abstract}

\section{Introduction}\label{s:intro}
Path-interaction processes have been introduced by Norris, Rogers and Williams during the late 80s in~\cite{NRW}. Since this period, they have been an intensive research area. Under the name of Brownian Polymers, Durrett and Rogers~\cite{DR} studied a family of self-interacting diffusions, as a model for the shape of a growing polymer.  Denoting by $X_t$ the location of the end of the growing polymer at time $t$, the process $X$ satisfies a stochastic differential equation driven by a Brownian motion,  with a drift term depending on its own occupation measure.   One is then interested in finding the scale for which the process converges to a non trivial limit. Later, another model of growing polymer has been introduced by Bena\"im, Ledoux and Raimond~\cite{BLR}, for which the drift term depends on its own empirical measure. Namely, they have studied  the following process living in a compact
smooth connected Riemannian manifold $M$ without boundary:
\begin{equation*}
\mathrm{d}X_t = \sum_{i=1}^N F_i(X_t)\circ\mathrm{d}B_t^i - \int_M \nabla_x W(X_t,y)\mu_t(\mathrm{d}y)\mathrm{d}t,
\end{equation*}
where $W$ is a (smooth) interaction potential, $(B^1,\cdots,B^N)$
is a standard Brownian motion on $\mathbb{R}^N$, $\mu_t =\frac{1}{t} \int_0^t \delta_{X_s} \rmd s$ and the symbol
$\circ$ stands for the Stratonovich stochastic integration. 
In the compact setting, they have shown that the asymptotic behaviour of the empirical measure of the process can be related to the analysis of some deterministic dynamical flow. Later, Bena\"im and Raimond~\cite{BR} gave sufficient conditions for the almost sure convergence of the empirical measure (again in the compact setting). More recently, Raimond~\cite{Ra} has generalized the previous study and has proved that for the solution of the SDE living on a compact manifold $$\rmd X_t = \rmd B_t - \frac{g(t)}{t} \int_0^t \nabla_x V(X_t, X_s) \rmd s \ \rmd t$$ unless $g$ is constant, the approximation of the empirical measure by a deterministic flow is no longer valid.

Similar questions have also been answered in the non-compact setting, that is $\Rn$. Chambeu and Kurtzmann~\cite{CK} have studied the ergodic behaviour of the self-interacting diffusion depending on the empirical mean of the process. They have proved, under some convexity assumptions (ensuring the non-explosion in finite time of the process), a convergence criterion for the diffusion solution to the SDE $$\rmd X_t = \rmd B_t - g(t) \nabla V\left(X_t - \frac{1}{t}\int_0^t X_s \rmd s \right)  \ \rmd t$$
where $g$ is a positive function. 
This model could represent for instance the behaviour of some social insects, as ants who are marking their paths with the trails' pheromones. This paper shows in particular how difficult is the study of general self-interacting diffusions in non-compact spaces as in~\cite{AK}, driven by the generic equation $$\rmd X_t = \rmd B_t - \frac{1}{t} \int_0^t \nabla_x V(X_t, X_s) \rmd s \ \rmd t.$$ Nevertheless, if the interaction function $V$ is symmetric and  uniformly convex, then Kleptsyn and Kurtzmann~\cite{kk-ejp} obtained the limit-quotient ergodic theorem for the self-attracting diffusion. Moreover, they managed to obtain a speed of convergence. As the results of this former paper are essential for the present work, we will explain them more precisely in \S\ref{ss:speed}. 
\medskip

Another problem related to this paper is the diffusion corresponding to McKean-Vlasov's PDE. This corresponds to the Markov process governed by the SDE
\begin{equation}
\rmd X_t = \rmd B_t - \nabla W * \nu_t (X_t)\rmd t
\end{equation}
where $\nu_t = \mathcal{L}(X_t)$, $W$ is a smooth convex potential and * stands for the convolution. The asymptotic behaviour of $X$ has been studied by various authors these last years, see for instance Cattiaux, Guillin and Malrieu~\cite{CGM}. It turns out that under some assumptions, the law $\nu_t$ converges to the (unique if $W$ is strictly convex) probability measure solution to the equation $\nu = \frac{1}{Z} e^{-2W*\nu}$ where $Z=Z(\nu)$ is the normalisation constant. In the latter paper, the authors use a particle system to prove both a convergence result (with convergence rate) and a deviation inequality for solutions of granular media equation when the interaction potential is uniformly convex at infinity. To this end, they use a uniform propagation of chaos property and a control in Wasserstein distance of solutions starting from different initial conditions. 

A related question to this problem concerns the exit-times from domains of attraction for the following motion
\begin{equation}\label{eq:MKV}
\rmd X_t^\sigma = \sigma \rmd B_t - \nabla V(X_t^\sigma) \rmd t - \nabla W*\nu_t (X_t^\sigma) \rmd t
\end{equation}
where $V$ is a potential, * stands for the convolution, $\nu_t = \mathcal{L}(X_t^\sigma)$ and $\sigma >0$. This was addressed by Herrmann, Imkeller and Peithmann~\cite{HIP}, who exhibited a Kramers' type law for the particle's exit from the potential's domains of attraction and a large deviations principle for the self-stabilizing (also named McKean-Vlasov) diffusion. To get this, they reconstructed the Freidlin-Wentzell theory for the self-stabilizing
diffusion. More precisely, they established a large deviations principle with a good rate function. The exit-problem for the McKean-Vlasov diffusion has also been already studied recently, without using the Freidlin-Wentzell method. In~\cite{T2011f}, Tugaut has analysed the exit-problem (time and location) in convex landscapes, showing the same result as Herrmann, Imkeller and Peithmann, but without reconstructing the proofs of Freidlin and Wentzell. He has then generalised very recently his results in the case of double-wells landscape in~\cite{JOTP}. In~\cite{Tug14c,PTRF}, Tugaut did not use large deviations principle but a coupling method
between the time-homogeneous diffusion $$\rmd X_t = \sigma \rmd B_t - \nabla V (X_t) \rmd t - \nabla W (X_t - m) \rmd t$$ (where $m$ is the unique point at which the vector field $\nabla V$ equals 0) and the McKean-Vlasov diffusion so that the results on the exit-time of $X$ can be used for the exit-time of the self-stabilizing diffusion~\eqref{eq:MKV}.

 The present paper also deals with the exit-time problem of a specific diffusion, driven by the SDE
\begin{equation}\label{eq:sid}
\mathrm{d}X_t = \sigma \rmd B_t - \left( \nabla V (X_t) + \frac{1}{t}\int_0^t \nabla W( X_t -  X_s) \rmd s \right) \rmd t, \quad X_0 = x\in \Rn
\end{equation}
where $V,W$ are two potentials and $\sigma  >0$. We could adapt the techniques introduced by Herrmann, Imkeller and Peithmann but only in the case of a convex gradient $V$. Our aim is to generalize the study also to non-convex potentials. In the present work, we will solve the exit-problem (time and location) for the diffusion~\eqref{eq:sid}. Indeed, the exit-location will be easily obtained once we know the exit-time. The motivation for the study of such a diffusion is twofold. First, we wish to obtain the basin of attraction  when the diffusion converges (if we know the speed of convergence, or at least a nice upper bound of it). And more challenging, our main aim consists in improving the simulated annealing method (even if we need a concave interaction in that case). This paper is a first  step in this direction.

In the following, for readability issue, we will omit the $sigma-$exponent for the process $X$ as well as the empirical measure $\mu_t$. Nevertheless, the reader has to keep in mind that the process and $\mu_t$ depend on $\sigma$.

\subsection{Some useful notations}
As usual, we denote by $\mathcal{M}(\Rn)$ the space of signed
(bounded) Borel measures on $\Rn$ and by $\mathcal{P}(\Rn)$ its
subspace of probability measures. We will need the following measure
space:
\begin{equation}
\mathcal{M}(\Rn;P):= \{\mu \in \mathcal{M}(\Rn); \int_{\Rn} P(|y|) \,
\vert\mu\vert(\mathrm{d}y)< +\infty\},
\end{equation}
where $\vert\mu\vert$ is the variation of $\mu$ (that is
$\vert\mu\vert := \mu^{+} + \mu^{-}$ with $(\mu^{+},\mu^{-})$ the
Hahn-Jordan decomposition of $\mu$: $\mu = \mu^+ -\mu^-$) and $P$ is some polynomial. Belonging to this space will enable us to
always check the integrability of $P$ (and therefore of $V,W$ and their   
derivatives thanks to the domination
condition~\eqref{domination}) with respect to the (random)
measures to be considered. We endow this space with the dual weighted supremum norm (or dual $P$-norm) defined for $\mu \in \mathcal{M}(\Rn;P)$ by
\begin{equation}
\vert \vert \mu\vert \vert_P := \sup_{\varphi\in \mathcal{C}(\Rn); \vert \varphi\vert
\leq P} \left\vert \int_{\Rn} \varphi\, \mathrm{d}\mu\right\vert = \int_{\Rn} P(|y|) \, |\mu|(\mathrm{d}y),
\end{equation}
where $\mathcal{C}(\Rn)$ is the set of continuous functions $\Rn \to \mathbb{R}$. Without any loss of generality, we suppose that $P(|x|)\ge 1$, so that $\|\mu\|_P \ge |\mu(\Rn)|$.
This norm makes $\mathcal{M}(\Rn;P)$ a Banach space. Next, we consider $\mathcal{P}(\Rn;P) =
\mathcal{M}(\Rn;P) \cap \mathcal{P}(\Rn)$. In the sequel, $(\cdot,\cdot)$ stands for the Euclidean
scalar product.

\begin{definition}
\label{stasta}
Let $d$ be any positive integer. Let $\mathcal{G}$ be a subset of $\mathbb{R}^{d}$ and let $U: \mathbb{R}^{d} \to \mathbb{R}^d$ be a vector field  satisfying some ``good assumptions''. For all $x\in\mathbb{R}^d$, we consider the dynamical system $\rho_t(x)=x+\int_0^tU\left(\rho_s(x)\right) \rmd{s}$. We say that the domain $\mathcal{G}$ is positively invariant for the flow generated by $U$ if the orbit $\left\{\rho_t(x)\,;\,t\in\mathbb{R}_+\right\}$ is included in $\mathcal{G}$ for all $x\in\mathcal{G}$.
\end{definition}

\subsection{Main results}

The precise assumptions on the potentials will be given later in Section~\ref{s:diffusion-linear}. 

The goal of this paper consists in finding some precise upper and lower bounds for the exit-time of some positively invariant domain. 
\begin{theorem}
\label{dell}
Let $\mathcal{D}$ be a domain that is positively invariant for the flow $x\mapsto -\nabla V(x) - \nabla W(x-m)$ and denote by $\tau$ the first time the process $X$ exits the domain $\mathcal{D}$. 
Let  $H:=\inf_{x\in\partial\mathcal{D}}\left(V(x)+W(x-m)-V(m)\right)$ be the exit cost from $\mathcal{D}$. Then $\displaystyle\mathbb{P}-\lim_{\sigma\to0}\frac{\sigma^2}{2}\log(\tau)=H$ that is for any $\delta>0$, we have
\begin{equation}
\label{eq:dell}
\lim_{\sigma\to0}\mathbb{P}\left(\exp\left\{\frac{2}{\sigma^2}\left(H-\delta\right)\right\}\leq\tau\leq\exp\left\{\frac{2}{\sigma^2}\left(H+\delta\right)\right\}\right)=1\,.
\end{equation}
\end{theorem}

From Theorem \ref{dell}, we immediately obtain the classical statement on the exit-location.
\begin{cor}
\label{hphp}
Under the same assumptions as the ones of Theorem \ref{dell}, if $\mathcal{N}\subset\partial\mathcal{D}$ is such that $\inf_{z\in\mathcal{N}}\left(V(z)+W(z-m)-V(m)\right)>H$, then
\begin{equation}
\label{eq:hphp}
\lim_{\sigma\to0}\mathbb{P}\left(X_{\tau}\in\mathcal{N}\right)=0\,.
\end{equation}

\end{cor}

\subsection{Outline}
Our paper is divided in two parts. First, Section~\ref{s:diffusion-linear} is devoted to the study of the self-interacting diffusion and more specifically, we will explain the former results of Kleptsyn and Kurtzmann~\cite{kk-ejp} and, more precisely, how we adapt them in our context. In particular, we will show that the empirical measure of the studied process converges almost surely with a rate  upper bounded independently of $\sigma$. After that, we will prove our main result, that is Theorem~\ref{dell}, in Section~\ref{s:proof}. To this aim, we will show that our process $X$ is close to the solution of a given deterministic flow in~\S\ref{s:majoration}. We then prove in~\S\ref{s:leaving} that the probability of leaving a positively invariant domain before the empirical mean remains stuck in the ball of center $m$ and radius $\kappa$ vanishes as $\sigma$ goes to zero. Finally, a coupling permits to conclude the proof of the main theorem in~\S\ref{ss:conclusion} and we give the proof of the corollary in~\S\ref{s:hphp}. 

\section{The self-interacting diffusion}\label{s:diffusion-linear}

 We assume that $V$ and $W$  satisfy the following set of hypotheses denoted by~\textbf{(H)}:
\begin{enumerate}
    \item[i)] (\textit{regularity and positivity}) $V \in \mathcal{C}^2(\mathbb{R}^d)$, $W\in
    \mathcal{C}^2(\mathbb{R}^d)$ and $V \geq 0, \; W \geq 0$;
    \item[ii)] (\textit{growth}) $V$ and $W$ have at most a polynomial growth: for some polynomial $P$, such that $P(|x|)\geq1$ for any $x\in\bRb^d$, we have $\forall x\in\Rn$:
    \begin{equation}\label{domination}
    \begin{aligned}
    |W(x)| +|\nabla W(x)| +  \|\nabla^2 W(x)\|  
    + |V(x)| +|\nabla V(x)| + \|\nabla^2 V(x)\|\leq P(|x|)
    \end{aligned}
    \end{equation}
    and $ \inf  \nabla^2 V\ge \rho>0$, $\inf \nabla^2 W\ge \alpha>0$,
    \begin{equation}\label{eq:growth}
    \Delta V(x)\leq aV(x) \,\, \text{and} \,\, \lim_{\vert x\vert\rightarrow \infty} \frac{\vert\nabla V(x)\vert^2}{V(x)} = \infty.
   \end{equation}
   
	\item[iv)] (\textit{curvature}) $V$ and $W$ are uniformly strictly convex functions. $V$ has a unique minimum in $m$;     
	\item[v)] (\textit{spherical symmetry}) $W(x)=G(|x|)$ for some function $G$ from $\bRb_+$ to $\bRb$.
\end{enumerate}
\begin{rem}
By the growth condition \eqref{eq:growth}, $\vert\nabla V\vert^2 - \Delta V$ is bounded by below.
\end{rem}
\begin{rem}
    Note, that without loss of generality we can choose polynomials $P$ to be of the form $P(|x|) = C(1 + |x|^k)$. Then the following property holds: there exists a constant $\gamma > 0$ such that $P(|x + y|) \leq \gamma(P(|x|) + P(|y|))$.
\end{rem}

Let us first show existence and uniqueness of the solution to the latter equation.
\begin{proposition}\label{p:existence-W} For any $x\in \mathbb{R}^d$, there exists a
unique global strong solution $(X_t,t\geq 0)$.
\end{proposition}
\begin{proof}
Local existence and uniqueness of the solution to \eqref{eq:sid} is standard (see for instance~\cite{DR}). 
We only need to prove here that $X$ 
does not explode in a finite time. Let us introduce the increasing sequence of stopping times $\tau_0 = 0$ and $$\tau_n :=
\inf\left\{t\geq \tau_{n-1}; \mathcal{E}_{t}(X_t) + \int_0^t \left\vert\nabla
\mathcal{E}_{s}(X_s)\right\vert^2 \mathrm{d}s > n\right\}$$ where $\mathcal{E}_{t}(X_t) := V(X_t) + \frac{1}{t} \int_0^t W(X_t-  X_s) \rmd s$. 
In order to show that the solution never explodes, we use the Lyapunov
functional $(x,t)\mapsto \mathcal{E}_t(x)$. As the process $(t,x)\mapsto\mathcal{E}_{\mu_t}(x)$ is of class $\mathcal{C}^2$ (in the space variable) and is a $\mathcal{C}^1$-semi-martingale
(in the time variable),
It\^o-Ventzell formula applied to
$(x,t)\mapsto\mathcal{E}_{t\wedge \tau_n}(x)$ implies
\begin{eqnarray}\label{eq:ito}
\mathcal{E}_{t\wedge \tau_n}(X_{t\wedge \tau_n}) = V(x)  + \int_0^{t\wedge \tau_n}
(\nabla \mathcal{E}_{s}(X_s), \mathrm{d}B_s) - \int_0^{t\wedge \tau_n} \left\vert\nabla \mathcal{E}_{s}(X_s)\right\vert^2 \mathrm{d}s\\
+ \frac{\sigma^2}{2}\int_0^{t\wedge \tau_n} \Delta \mathcal{E}_{s}(X_s)\mathrm{d}s
- \int_0^{t\wedge \tau_n}  \int_0^s  W(X_s-  X_u)\rmd u \frac{\mathrm{d}s}{s^2}.\notag
\end{eqnarray}
We note that $\int_0^{t\wedge \tau_n} (\nabla \mathcal{E}_{s}(X_s),
\mathrm{d}B_s)$ is a true martingale.
We then get 
\begin{equation*}
\mathbb{E}\mathcal{E}_{t\wedge \tau_n}(X_{t\wedge \tau_n})
\leq V(x) + a \int_0^{t}
\mathbb{E}\mathcal{E}_{s\wedge \tau_n}(X_{s\wedge \tau_n})
\mathrm{d}s.
\end{equation*}
So, Gronwall's lemma leads to: $$\mathbb{E}V(X_{t\wedge \tau_n})\le \mathbb{E}\mathcal{E}_{t\wedge
\tau_n}(X_{t\wedge \tau_n}) \leq V(x)e^{at}.$$
As $\underset{\vert x\vert\rightarrow \infty}{\lim}V(x) =\infty$, the process $(X_t,t\geq 0)$ does not explode in a finite time and there exists a global strong solution.
\end{proof}

\begin{rem}
A large family of path-dependent process has been studied by Saporito, see for instance~\cite{yuri}. He proves, with his co-authors, existence and uniqueness of such processes. The difference with our process is that we normalize the occupation measure.
\end{rem}

\subsection{Speed of convergence}\label{ss:speed}
Our results are based on the paper of Kleptsyn and Kurtzmann~\cite{kk-ejp}. More precisely, they have proved the following
\begin{theorem}\cite[Theorem 1.6]{kk-ejp}\label{t:main-2}
Let $X$ be the solution to the equation \eqref{eq:sid} with $\sigma=\sqrt{2}$.

Suppose, that $V\in \mathcal{C}^2(\Rn)$ and $W\in \mathcal{C}^2(\Rn)$, satisfy \textbf{(H)} and either $V$ or $W$ is uniformly strictly convex that is there exists $C>0$ such that
$$
\forall x\in \Rn, \forall v\in \Sn , \, \left.\frac{\partial^2 V}{\partial v^2}\right|_{x}\ge C\quad \text{ or } \quad  \forall x\in \Rn, \forall v\in \Sn,\ \left.\frac{\partial^2 W}{\partial v^2}\right|_{x}\ge C.
$$
Then there exists a unique density $\rho_{\infty}:\Rn\to\bbR_+$, such that almost surely
$$
\mu_t = \frac{1}{t}\int_0^t \delta_{X_s}\mathrm{d}s \xrightarrow[t\to+\infty]{*-weakly}
\rho_{\infty}(x) \, \mathrm{d}x.
$$
\end{theorem}
Moreover, if $V$ is symmetric with respect to some point $q$, then the corresponding density $\rho_\infty$ is also symmetric with respect to the same point $q$. Remark that the density $\rho_{\infty}$ is the same limit density as in the result of~\cite{CMV}, uniquely defined by the following property: $\rho_{\infty}$ is a positive function, proportional to~$e^{-(V+W*\rho_{\infty})}$.

 And what is more important, Kleptsyn and Kurtzmann obtained a speed of convergence in the following way. First, let us recall the definition of the Wasserstein distance.
\begin{definition} 
  For $\mu_1,\mu_2 \in \mathcal{P}(\Rn;P)$, the quadratic Wasserstein distance is defined as $$\mathbb{W}_2(\mu_1,\mu_2) := \left(\inf \{ \mathbb{E}(|\xi_1-\xi_2|^2)\}\right)^{1/2},$$ where the infimum is taken over all the random variables such that $\{$law of $\xi_1\}= \mu_1$ and $\{$law of $\xi_2\}= \mu_2$. This corresponds to the minimal $L^2$-distance taken over all the couplings between $\mu_1$ and $\mu_2$.
  
  Similarly, the Wasserstein distance $\mathbb{W}_{2k}$ is defined as 
  $$\mathbb{W}_{2k}(\mu_1,\mu_2) := \left(\inf \{ \mathbb{E}(|\xi_1-\xi_2|^{2k})\}\right)^{1/(2k)}.$$
\end{definition}

More precisely, in \cite[Theorem 1.12]{kk-ejp}, it is proved the existence of a constant $a>0$ such that almost surely, for $t$ large enough one has 
\begin{equation*}
\mathbb{W}_2(\mu_t^c, \rho_\infty) = O(\exp\{-a\sqrt[2k+1]{\log t}\})\,,
\end{equation*}
where $2k$ is the degree of the polynomial $P$, $\mu_t^c$ is the translation of the empirical measure $\mu_t$ such that $\E (\mu_t^c)=0$ and $\mathbb{W}_2$ is the quadratic Wasserstein distance.

Of course, in the case $W(x) = \alpha \frac{|x|^2}{2}$, the polynomial $P$ corresponds to the growth of $V$, and we have to replace the Brownian motion by the rescaled Brownian motion $\sigma B_t$. So that the density $\rho_{\infty}$ is  uniquely defined by the following property: $\rho_{\infty}$ is a positive function, proportional to~$e^{-2(V+W*\rho_{\infty})/\sigma^2}$. Let us sketch the proof of these results and explain how $\sigma$ appears  here.

First, note that the empirical measure $\mu_t=\frac{1}{t}\int_0^t \delta_{X_s} \rm ds$ evolves very slowly. Indeed,  choose a deterministic sequence of times $T_n\to+\infty$, with
$T_n\gg T_{n+1}-T_{n}\gg 1$, and consider the behaviour of the
measures~$\mu_{T_n}$. As $T_n\gg T_{n+1}-T_n$, it is natural to
expect  that the
empirical measures $\mu_t$ on the interval $[T_n,T_{n+1}]$ almost
do not change and thus stay close to $\mu_{T_n}$. So we can approximate, on this
interval, the solution $X_t$ of~\eqref{eq:sid} with $\sigma=\sqrt{2}$ by the solution of the same equation
with $\mu_t\equiv \mu_{T_n}$:
$$
\mathrm{d}Y_t = \sigma\, \mathrm{d}B_t - (\nabla V+\nabla W*\mu_{T_n})(Y_t) \, \mathrm{d}t,\quad t\in [T_n,T_{n+1}],
$$
in other words, by a Brownian motion in a potential~$V+W*\mu_{T_n}$ that does not depend on time.

On the other hand, the series of general term $T_{n+1}-T_n$ increases. So, using Birkhoff Ergodic Theorem,  
we see that the (normalized)
distribution $\mu_{[T_n,T_{n+1}]}$ of values of $X_t$ on these
intervals becomes (as $n$ increases) close to the equilibrium measures
$\Pi(\mu_{T_n})$ for a Brownian motion in the potential
$V+W*\mu_{T_n}$, where 
$$
\Pi(\mu)(\mathrm{d}x) := \frac{1}{Z(\mu,\sigma)} e^{-2(V+W*\mu)(x) /\sigma^2} \,
\mathrm{d}x,\quad Z(\mu,\sigma):=\int_{\Rn} e^{-2(V+W*\mu)(x) /\sigma^2} \,
\mathrm{d}x.$$
As
$$
\mu_{T_{n+1}} = \frac{T_n}{T_{n+1}} \,\mu_{T_n} +
\frac{T_{n+1}-T_n}{T_{n+1}} \,\mu_{[T_n,T_{n+1}]},
$$
 we then have
$$
\mu_{T_{n+1}} \approx \frac{T_n}{T_{n+1}} \,\mu_{T_n} +
\frac{T_{n+1}-T_n}{T_{n+1}} \,\Pi(\mu_{T_n}) = \mu_{T_n} +
\frac{T_{n+1}-T_n}{T_{n+1}} (\Pi(\mu_{T_n})-\mu_{T_n}),
$$
and
$$
\frac{\mu_{T_{n+1}}-\mu_{T_n}}{T_{n+1}-T_n} \approx
\frac{1}{T_{n+1}}(\Pi(\mu_{T_n})-\mu_{T_n}).
$$
This motivates Kleptsyn and Kurtzmann to approximate the behaviour of the measures $\mu_t$
by trajectories of the flow (on the infinite-dimensional space of
measures)
\begin{equation}\label{eq:9-3/4}
\dot{\mu}=\frac{1}{t}(\Pi(\mu)-\mu),
\end{equation}
or after a logarithmic change of variable $\theta=\log t$,
\begin{equation}\label{eq:flow}
\mu'=\Pi(\mu)-\mu.
\end{equation}

Indeed, choose an appropriate interval $[T_n, T_{n+1})$. On this interval, fix the empirical measure $\mu_t$ at $\mu_{T_n}$. Then construct a new process $Y$, coupled with $X$ (the coupling is such that $X$ and $Y$ are driven by the same Brownian motion), such that for all $t\in [T_n, T_{n+1})$, we have $$\mathrm{d}Y_t = \sigma \mathrm{d}B_t - (\nabla V(Y_t) + \nabla W*\mu_{T_n}(Y_t)) \mathrm{d}t.$$
This new process has  two advantages. First, it is Markovian (and its invariant probability measure is $\Pi(\mu_{T_n})(\mathrm{d}x) = \frac{1}{Z} e^{-2(V+W*\mu_{T_n})(x)/\sigma^2} \mathrm{d}x$), and so is easier than $X$ to study. Second, its evolution is very close to the evolution of the desired $X$. Indeed, we will use $Y$ to prove that the transport distance between the empirical measure on $[T_n, T_{n+1}]$, denoted by $\mu_{[T_n, T_{n+1}]}$ in~\cite[Proposition 3.2]{kk-ejp}, and the probability measure $\Pi(\mu_{T_n})$ (both measures being centered in $c_{T_n}$) is controlled by $T_n^{-\frac{1}{3}\min (8C_W,1/5d)}$ and so, this distance vanishes as $n\to +\infty$. This has been done in \cite[\S 3.1.1]{kk-ejp}.\\
After that, remark that if a.s. the empirical measure $\mu_t$ converges weakly* to $\mu_\infty$, then for $t$ large enough, the process $X$ shall be very close to $Z$, defined by $$\mathrm{d}Z_t = \sigma \mathrm{d}B_t - (\nabla V+\nabla W*\mu_{\infty})(Z_t) \mathrm{d}t.$$
The process $Z$ is obviously Markovian and the ergodic theorem can be applied: 
$$\frac{1}{t} \int_0^t \delta_{Z_s} \mathrm{d}s \underset{t\to +\infty}{\longrightarrow} \Pi(\mu_\infty)\quad \text{a.s.}$$ for the weak* convergence of measures. So when the limit $\mu_\infty$ exists, it satisfies
\begin{equation}\label{eq:mu_infty_and_Pi}
\mu_\infty = \Pi(\mu_\infty)\,.
\end{equation}

This explains the idea of introducing the dynamical system \mbox{$\dot{\mu} = \Pi(\mu)-\mu$} (after the time-shift $t\mapsto e^t$ in order to work with a time-homogeneous system) defined on the set of probability measures that are integrable for the polynomial $P$. Note that, instead of considering the latter dynamical system, Kleptsyn and Kurtzmann work with its discretized version, with the knots chosen at the moments $T_n$. They then prove, in~\cite[Proposition 3.5]{kk-ejp}, that the transport distance between the deterministic trajectory induced by the smoothened (discrete) dynamical system and the (centered) random trajectory $\mu_{T_n}$ is controlled and decreases to 0. This has been done in~\cite[\S 3.1.2]{kk-ejp}.\\
Next, it remains to show that the free energy  between this (centered) deterministic trajectory and the set of translates of $\rho_\infty$ goes to 0. We recall that  for the dynamics in presence of an exterior potential $V$, the free energy function is
$$
\mF_{V,W}(\mu) := -\frac{\sigma^2}{2} \mathcal{H}(\mu) + \int_{\Rn} V(x) \mu(x)\, \mathrm{d}x \\
+ \frac{1}{2} \iint_{\Rn\times\Rn} \mu(x) W(x-y) \mu(y)\, \mathrm{d}x\, \mathrm{d}y.
$$
and consider $\mF_{V+W*\mu}= -\frac{\sigma^2}{2} \mathcal{H}(\mu) + \int_{\Rn} (V(x)+W*\mu (x)) \mu(x)\, \mathrm{d}x$ for the energy of ``small parts'', where the entropy of the measure $\mu$ is  
\begin{equation}\label{eq:entropy}
\mathcal{H}(\mu) := -\int_{\Rn} \mu(x) \log\mu(x) \mathrm{d}x.
\end{equation}

As the free energy is controlled by the quadratic Wasserstein distance $\mathbb{W}_2$, this implies that the transport distance between the two previous quantities decreases, as asserted in \cite[Proposition 3.6]{kk-ejp}.\\ 
To conclude, it remains to put all the pieces together and use the triangle inequality: $\mathbb{W}_2 (\mu_t^c,\rho_\infty)$ is upper bounded by the sum of three distances, involving the flow  $\Phi_n$ induced by the discretization of the dynamical system $\dot{\mu} = \Pi(\mu) - \mu$ on the interval $[T_n, T_{n+1})$, for $n$ large enough. The first term of the summation bound will be $\mathbb{W}_2 (\mu_t^c, \Pi(\mu_{T_n}^c))$, the second one $\mathbb{W}_2(\Pi(\mu_{T_n}^c), \Phi_n^n(\mu_{T_n}^c))$ and the third one $\mathbb{W}_2 (\Phi_n^n(\mu_{T_n}^c),\rho_\infty)$.

Finally, the previous decrease estimates will allow Kleptsyn and Kurtzmann to show the convergence of the center, after having made the appropriate choice $T_n = n^{3/2}$.

\subsection{The speed of convergence for the solution of \texorpdfstring{\eqref{eq:sid}}{(\ref{eq:sid})}}
For this paper, the corresponding result of \cite[Theorem 1.12]{kk-ejp} is the following
\begin{theorem}
\label{thm:KK12}
Let $X$ be the solution to the equation \eqref{eq:sid}. There exists a constant $a>0$ such that almost surely, we have for $t$ large enough $\mathbb{W}_{2k}(\mu_t, \rho_\infty) = O(\exp\{-a\sqrt[2k+1]{\log t}\})$, where $2k$ is the degree of the polynomial $P$, $\mu_t$ is  the empirical measure and $\mathbb{W}_{2k}$ is the $2k$-Wasserstein distance, that is the minimal $L^{2k}$-distance taken over all the couplings between $\mu_t$ and $\rho_\infty$ that is the unique density proportional to $e^{-\frac{2}{\sigma^2}(V+W*\rho_\infty)}$.
\end{theorem}
To prove this result, we mimic the proof of \cite[Theorem 1.12]{kk-ejp}, and show that the speed of convergence is less than the one of the case corresponding to $\sigma =1$. We will not reproduce it here. We will only show how we handle this inequality for the following particular result, as it is representative of the difficulty and shows how $\sigma$ appears in the calculation.

Let us prove for instance the exponential decrease for the centered measure $\Pi(\mu)$ and show that, as $\sigma^2 \ll 1$, we can obtain a lower bound of the speed of convergence that does not depend on $\sigma$. We have seen previously that the centered $\mu^c$ has the same asymptotic behaviour as $\Pi(\mu)(\cdot + c_\mu)$, up to a time-scale. An important step to prove the latter result consists in estimating the behaviour of the centered measures $\mu_t^c$. More precisely, one has to prove that the tail of these measures are exponentially decreasing. This is why we have to introduce the following sets:
 \begin{definition}\label{eq:K-alpha-C}
Let $\alpha,C>0$ be given. Define 
\begin{eqnarray}
K_{\alpha,C}^0 &:=& \{\mu \in \mcP(\Rn); \quad \forall R>0, \, \mu(\{y; |y|>R\})<Ce^{-\alpha R}\},\\
K_{\alpha,C} &:=& \{\mu \in \mcP(\Rn); \quad \mu^c \in K_{\alpha,C}^0 \}.
\end{eqnarray}
\end{definition}
We can now prove the exponential decrease of $\Pi(\mu)(\cdot + c_\mu)$.
\begin{proposition} \label{prop:exp-decrease}
There exist $C_W,C_\Pi>0$, two constants  independent of $\sigma$, such that for all $\mu \in \mathcal{P}(\mathbb{R};P)$, we have $\Pi(\mu)(\cdot +c_\mu) \in K^0_{C_W,C_\Pi}$ where $c_\mu$ is defined by the equation $(\nabla V + \nabla W *\mu )(c_\mu) = 0$.
\end{proposition}
\begin{proof}
Let us fix $R>0$. Note first that, imposing a condition $C_\Pi\ge e^{2C_W}$, we can restrict ourselves only to $R\ge 2$: for $R<2$, the estimate is obvious. 

The measure $\Pi(\mu)$ has the density $\frac{1}{Z(\mu,\sigma)} e^{-2(V+W*\mu)(x)/\sigma^2}$. To avoid working with the normalization constant $Z(\mu,\sigma)$, we will prove a stronger inequality, that is 
\begin{equation}\label{eq:exp-dec}
\Pi(\mu)(|x-c_\mu|\ge R)\le
C_\Pi e^{-C R}\cdot \Pi(\mu)(|x-c_\mu|\le 2),
\end{equation}
which is equivalent to $$\int_{|x-c_\mu|\ge R} e^{-2(V+W*\mu)(x) /\sigma^2} \mathrm{d}x \le C_\Pi e^{-C R} \int_{|x-c_\mu|\le 2} e^{-2(V+W*\mu)(x)/\sigma^2} \mathrm{d}x.$$ 
We use the polar coordinates, centered at the center $c_\mu$, and so we want to prove that 
\begin{eqnarray*}
& & \int_{\mathbb{S}^{d-1}} \int_R^{+\infty} e^{-2(V+W*\mu) (c_\mu + \lambda v) / \sigma^2}\lambda^{d-1}\mathrm{d}\lambda \mathrm{d}v \\
& & \quad \le   C_\Pi e^{-C R} \int_{\mathbb{S}^{d-1}} \int_0^2 e^{-2(V+W*\mu) (c_\mu + \lambda v)/\sigma^2}\lambda^{d-1}\mathrm{d}\lambda \mathrm{d}v.
\end{eqnarray*}
It suffices to prove such an inequality ``directionwise'': for all $v\in \mathbb{S}^{d-1}$, for all $R\ge 2$ 
$$\int_R^{+\infty} e^{-2(V+W*\mu) (c_\mu + \lambda v) /\sigma^2}\lambda^{d-1}\mathrm{d}\lambda \le C_\Pi e^{-C R} \int_0^2 e^{-2(V+W*\mu) (c_\mu + \lambda v)/\sigma^2}\lambda^{d-1}\mathrm{d}\lambda.$$ But from the uniform convexity of $V$ and $W$ and the definition of the center, the function $f(\lambda) = 2(V+ W*\mu) (c_\mu + \lambda v)$ satisfies $f'(0) =0$ and $\forall r>0$, $f''(r)\ge C$. Hence, $f$ is monotone increasing on $[0,+\infty)$, and in particular, 
\begin{equation}\label{eq:coord-polaires}
\int_0^2 e^{-f(\lambda)/\sigma^2} \lambda^{d-1} \mathrm{d}\lambda \ge e^{-f(2)/\sigma^2} \int_0^2 \lambda^{d-1}\mathrm{d}\lambda =: C_1 e^{-f(2)/\sigma^2}.
\end{equation}
On the other hand, for all $\lambda \ge 2$, $f'(\lambda) \ge f'(2)\ge 2C$, and thus $f(\lambda)\ge 2C (\lambda -2) + f(2)$. Hence, as $\sigma^2\ll 1$, we have 
\begin{eqnarray}
\int_R^{+\infty} e^{-f(\lambda)/\sigma^2} \lambda^{d-1} \mathrm{d}\lambda & \le & e^{-f(2)/\sigma^2} \int_R^{+\infty} \lambda^{d-1} e^{-2C(\lambda-2)/\sigma^2} \mathrm{d}\lambda \notag\\
& \le & e^{-\frac{f(2)}{\sigma^2}} \int_R^{+\infty} \lambda^{d-1} e^{-2C(\lambda-2)} \mathrm{d}\lambda.
\end{eqnarray}
So finally, using that $\sigma\ll 1$
\begin{eqnarray}\label{eq:coord-polaires-2}
\int_R^{+\infty} e^{-f(\lambda)/\sigma^2} \lambda^{d-1} \mathrm{d}\lambda \le C_2 R^{d-1} e^{-2C R} \cdot e^{-\frac{f(2)}{\sigma^2}} \le C_3 e^{-CR} \cdot e^{-\frac{f(2)}{\sigma^2}}.
\end{eqnarray}
Comparing~\eqref{eq:coord-polaires} and~\eqref{eq:coord-polaires-2}, we obtain the desired exponential decrease. Indeed, as $\sigma$ only appears in the exponential under the form $\exp (-const /\sigma^2)$, we can upper bound the previous quantity by $e^{-const}$, so that the constant $C_W$ is upper bounded by a constant that is independent of $\sigma$. By an abuse of notation, we call again this constant $C_W$.
\end{proof}
\begin{rem}
This result corresponds to \cite[Proposition 2.9]{kk-ejp} for the case $\sigma =1$.
\end{rem}

\subsection{Small-noise limit}
In the following, we remind the reader that we don't emphasize the dependence on $\sigma$, but it will appear everywhere in the computations.  

As was mentioned above, the invariant probability measure of self-interacting diffusion is the unique solution to the equation
\begin{equation*}
    \rho_\infty = \Pi(\rho_\infty),
\end{equation*}
where $\Pi(\mu)(x) =  e^{-2(V+W*\mu)(x) /\sigma^2} / \int e^{-2(V+W*\mu)(z) /\sigma^2} \rmd z$. The same invariant probability measure appears in the self-stabilizing diffusion, small-noise limit of which was studied in \cite{HT-EJP}. There, authors studied the case of double-wells potentials which is more general then our diffusion. In this paper the result, that can be transformed in our context as following, was proved. If the moments of invariant probability measures $\rho_\infty$ are uniformly bounded with respect to $\sigma$, then $\delta_m$ is the weak-* limit of $\rho_\infty$ with $\sigma \to 0$ a.s. Note, that indeed, moments of $\rho_t$ are uniformly bounded, since $\mu_t \in K_{\alpha, C}$ for any $t > 0$ and some $\alpha, C$ that do not depend on $\sigma$. 

Thus, consider the following deterministic time, representing the time of stabilization of the occupation measure, if it occurs, around its supposed limit $\delta_m$: 
\begin{equation}
\label{def:Tkappa}
T_\kappa(\sigma):=\inf\left\{t_0\geq0\,\,:\,\,\forall t\geq t_0,\,\mathbb{E}\left(\mathbb{W}_{2k}\left(\mu_t; \delta_m\right)\right)\leq\kappa\right\}\,.
\end{equation}

First, let us discuss why expectation $\E \left(\mathbb{W}_{2k}\left(\mu_t; \delta_m\right)\right)$ exists in the first place. To show that, we use the fact that $\mu_t \in K_{\alpha, C}$ almost surely and get
\begin{equation*}
    \mathbb{W}_{2k}\left(\mu_t; \delta_m\right) \leq \left( 2^{2k - 1} \int |x|^{2k} \mu_t(\rmd x) + 2^{2k - 1} |m|^{2k}\right)^{1/(2k)} \leq \text{Const},
\end{equation*}
where the last constant depends only on $\alpha, C, m$ and $k$. Therefore, since the random variable is bounded by a constant almost surely, expectation exists.

Second, let us show that the definition of the time $T_\kappa(\sigma)$ makes sense. Indeed, 
\begin{equation*}
    \E\mathbb{W}_{2k}\left(\mu_t; \delta_m\right) \leq \E\mathbb{W}_{2k}\left(\mu_t; \rho_\infty\right) + \E\mathbb{W}_{2k}\left(\rho_\infty; \delta_m \right) \xrightarrow[\substack{t \to \infty \\ \sigma \to 0}]{} 0,
\end{equation*}
where the limit is not just iterated, but holds for the pair $(t, \sigma)$, since the speed of convergence of $\mu_t$ towards $\rho_{\infty}$ in time does not depend on $\sigma$, which was shown in Theorem \ref{thm:KK12}. Therefore, for any $\kappa > 0$ we can find $\sigma_0$ small enough and $t_0$ big enough such that $T_\kappa(\sigma) < T_\kappa < \infty$ for any $\sigma < \sigma_0$, which does not only prove existence and finiteness of $T_\kappa(\sigma)$, but also its uniformness with respect to $\sigma$.

\section{Proof of Theorem \ref{dell}} \label{s:proof}
In this section, we prove our main result. First, we show that the process $X$ solution to \eqref{eq:sid} is close to the solution of a deterministic flow $(\psi_t)_{t\ge 0}$ in~\S\ref{s:majoration}. Using that, we prove in~\S\ref{s:leaving} that the probability of leaving a positively invariant domain before the empirical mean remains stuck in the ball of center $m$ and radius $\kappa$ vanishes as $\sigma$ goes to zero. Then, we consider the coupling between the studied diffusion and the one where the empirical measure is frozen to $\delta_m$ and we show that these diffusions are close in~\S\ref{ss:coupling}. We conclude the proof in~\S\ref{ss:conclusion}.
 
\subsection{Upper bound}\label{s:majoration}
We remind the reader that in this work, the noise vanishes. Consequently, it is natural to introduce the deterministic flow $(\psi_t)_t$ defined by the following
\begin{equation}
\dot{\psi}_t = -\nabla V(\psi_t) -\frac{1}{t}\int_0^t \nabla W(\psi_t-\psi_s ) \rmd s, \quad \quad \psi_0 = x_0.
\end{equation}
We will show that $X_t$ and $\psi_t$ are uniformly close while the noise goes to zero. Namely
\begin{proposition} 
For any $\xi>0$ and for any $T>0$, we have:
\begin{equation}
\label{apple}
\lim_{\sigma\to 0}\mathbb{P}\left(\sup_{t\in [0;T]}\left|\left|X_t-\psi_t(x_0)\right|\right|^2>\xi\right)=0\,.
\end{equation}
\end{proposition}

 \begin{proof}

First of all, we fix some $\xi$ and introduce the following stopping time $\mathcal{T} := \inf\{t: |X_t^\sigma - \psi_t|^2 \geq \xi\}$. We apply Itô formula and get the following result, for  $\omega \in \{\mathcal{T} > t\}$ (the choice of this event will be clear further) : 
\begin{equation*}\label{eq:X_t-psi_t}
\begin{aligned}
    |X_t - \psi_t|^2 &= 2\int_{0}^t (X_s - \psi_s, \rmd{X_s}-\rmd{\psi_s}) + d \sigma^2 t \\
    & \leq d\sigma^2t -2C_V \int_0^t |X_s - \psi_s|^2 \rmd s + \sigma \int_0^t (X_s - \psi_s , \rmd B_s) \\
    & \quad - \int_0^t\frac{1}{s} \int_0^s (X_s - \psi_s, \nabla W(X_s - X_z) - \nabla W(\psi_s - \psi_z))\rmd z\rmd s.
 \end{aligned}
\end{equation*}   
Let $\text{Lip}_{\nabla W}^{K^\prime}$ be a Lipschitz constant of $\nabla W$ inside the following compact\\ $K^\prime:= \{x: |x - \psi_t|^2 \leq \xi \text{, for some } t > 0\}$ (due to our assumptions, this set is indeed a compact at least for small $\xi$, which we can decrease without loss of generality), and $C_V$ is the convexity constant of $V$. We thus have   
\begin{eqnarray}\label{eq:X_t-psi_tbis}
    |X_t - \psi_t|^2    
    & \leq & d\sigma^2 t - 2C_V\int_0^t |X_s - \psi_s|^2\rmd s + \sigma \int_0^t(X_s - \psi_s , \rmd B_s)\notag \\
    &+& \text{Lip}_{\nabla W}^{K^\prime}\int_0^t\frac{1}{s}\int_0^s \big(|X_s - \psi_s|^2 + |X_s - \psi_s|\cdot|X_z - \psi_z|\big) \rmd z\notag \\
    & \leq & d\sigma^2 t - 2C_V\int_0^t |X_s - \psi_s|^2\rmd s + \sigma \int_0^t(X_s - \psi_s , \rmd B_s) \\
    &+ & \frac{\text{Lip}_{\nabla W}^{K^\prime}}{2}\int_0^t\frac{1}{s}\int_0^s \big(3|X_s - \psi_s|^2 + |X_z - \psi_z|^2\big) \rmd z.\notag
\end{eqnarray}
Note then, that by the BDG inequality we get for some constant $C > 0$:
\begin{equation*}
\begin{aligned}
    \E \left(\sup_{[0,t \wedge \mathcal{T}]} \left|\sigma \int_0^s (X_z-\psi_z, \rmd B_z )\right| \right) &\le C \sigma^2 \E \sqrt{\int_0^{t\wedge \mathcal{T}} |X_s - \psi_s|^2 \rmd s}\\
    & \le C \sigma^2 \sqrt{\int_0^t \E \left(\sup_{z \in [0, s \wedge \mathcal{T}]}(|X_z - \psi_z|^2)\right) \rmd s}. 
\end{aligned}
\end{equation*}

Let us consider the following random variable $\sup_{s \in [0; t \wedge \mathcal{T}]} |X_s - \psi_s|^2$. The fact that we consider the supremum before time $t \wedge \mathcal{T}$ gives us that for any $\omega$ we consider only such $s$, that $s \leq \mathcal{T}(\omega)$, which in turn means that we can apply estimation \eqref{eq:X_t-psi_tbis} for any $s \in [0, t\wedge \mathcal{T}]$. We also remind that $t \leq T$ and derive:
\begin{equation*}
\begin{aligned}
    \E \left(\sup_{s \in [0; t \wedge \mathcal{T}]} |X_s - \psi_s|^2\right) &\leq d\sigma^2T+ C \sigma^2 \sqrt{\int_0^t \E \left(\sup_{z \in [0 , s \wedge \mathcal{T}]} (|X_z - \psi_z|^2)\right) \rmd z} \\
    &\quad + 2 \text{Lip}_{\nabla W}^{K^\prime} \int_0^t \E \left(\sup_{z \in [0, s \wedge \mathcal{T}]} |X_z - \psi_z|^2\right) \rmd s \\
    &\leq  d\sigma^2T + \frac{C \sigma^2}{2} \left[1 +  T\E\left( \sup_{s \in [0 , t \wedge \mathcal{T}]} (|X_s - \psi_s|^2)\right)  \right] \\
    & \quad + 2 \text{Lip}_{\nabla W}^{K^\prime} \int_0^t \E \left(\sup_{z \in [0, s \wedge \mathcal{T}]} |X_z - \psi_z|^2\right) \rmd s,
\end{aligned}
\end{equation*}
where in the last inequality we used $\sqrt{x} \leq (1 + x)/2$. Now, if we denote\\ $u_t := \E \left(\sup_{s \in [0; t \wedge \mathcal{T}]} |X_s - \psi_s|^2\right)$, we have
\begin{equation*}
    u_t \leq \frac{1}{1 - C T \sigma^2/2} \left(\frac{2Td + C}{2}\sigma^2 + 2 \text{Lip}_{\nabla W}^{K^\prime} \int_0^t u_s \rmd s \right),
\end{equation*}
for small enough $\sigma$ (such that $1 - C T \sigma^2/2 > 0$). Thus, using Gr\"onwall lemma, we get
\begin{equation}\label{eq:ut_bound}
    u_t \leq \frac{(2Td + C) \sigma^2}{2(1 - C T \sigma^2/2)}\exp\left\{\frac{\text{Lip}_{\nabla W}^{K^\prime} }{1 - C T \sigma^2/2} T\right\} = O(\sigma^2).
\end{equation}
This in particular means, that $\E\left( \sup_{s \in [0; T \wedge \mathcal{T}]} |X_s - \psi_s|^2 \right) \leq O(\sigma^2)$. Nevertheless, to show the necessary result, we have to get rid of the stopping time $\mathcal{T}$ in the previous equation. It is sufficient to show, that $\mathbb{P} (\mathcal{T} \leq T) \xrightarrow[\sigma \to 0]{} 0$. 

Indeed, by its definition, $\mathcal{T}$ is the first time when the difference $|X_t - \psi_t|^2$ reaches $\xi$. But under the assumption $\mathcal{T} \leq T$ and due to \eqref{eq:ut_bound}, by decreasing $\sigma$ we can control $|X_t - \psi_t|^2$ and make it small enough, such that $|X_{\mathcal{T}} - \psi_{\mathcal{T}}|^2 < \xi$ (in some sense), which contradicts the definition of $\mathcal{T}$. Rigorously, 
\begin{equation*}
    \mathcal{T} < T \Rightarrow \sup_{[0, \mathcal{T}\wedge T]} |X_s - \psi_s|^2 = \sup_{[0, \mathcal{T}]} |X_s - \psi_s|^2 \ge \xi.
\end{equation*}
Thereby, 
\begin{equation*}
    \mathbb{P} (\mathcal{T} < T) \leq \mathbb{P}(\sup_{[0, \mathcal{T}\wedge T]} |X_s - \psi_s|^2 \ge \xi) \leq O(\sigma^2),
\end{equation*}
by Markov inequality.

To conclude the proof of the Proposition, we consider
\begin{equation*}
\begin{aligned}
    \mathbb{P}\left(\sup_{t\in [0;T]}|X_t-\psi_t(x_0)|^2>\xi\right) & \leq \mathbb{P}\left(\sup_{t\in [0;T]}|X_t-\psi_t(x_0)|^2>\xi, \mathcal{T} > T\right) \\
    & \quad\quad + \mathbb{P} \left( \mathcal{T} \leq T \right) \\
    & \leq \mathbb{P}\left(\sup_{t\in [0;T \wedge \mathcal{T}]}|X_t-\psi_t(x_0)|^2>\xi \right) + O(\sigma^2) \\
    & \leq O(\sigma^2),
\end{aligned}
\end{equation*}
by Markov inequality and \eqref{eq:ut_bound}, which completes the proof.
\end{proof}

\subsection{Probability of leaving before \texorpdfstring{$T_{\kappa}(\sigma)$}{Tkappasigma}}\label{s:leaving}

Remind that we denoted in~\eqref{def:Tkappa} by $T_{\kappa}(\sigma)$ the first time at which the expectation of the $2k$-Wasserstein distance between the occupation measure of the process and $\delta_m$ is smaller than $\kappa$. By $\mathbb{B}(m;\kappa)$, we denote the ball of center $\delta_m$ and radius $\kappa$ for $\mathbb{W}_{2k}$.
\begin{proposition}
\label{craie}
We put $\tau:=\inf\{t\geq0\,\,:\,\,X_t\notin\mathcal{D}\}$ where $\mathcal{D}$ is a domain that is positively invariant for the flow $x\mapsto-\nabla V(x)-\nabla W(x-m)$. For any $\kappa>0$, 
\begin{equation}
\label{apple2}
\lim_{\sigma\to0}\mathbb{P}\left(\tau\leq T_\kappa(\sigma)\right)=0\,.
\end{equation}
\end{proposition}

\begin{proof}
First, we remind that if $\sigma$ is small enough, $T_\kappa(\sigma)\leq T_\kappa$ where $T_\kappa$ does not depend on $\sigma$. Moreover, $T_\kappa$ and $T_\kappa(\sigma)$ are deterministic. As a consequence, we have
\begin{equation*}
\mathbb{P}\left(\tau\leq T_\kappa(\sigma)\right)\leq\mathbb{P}\left(\tau\leq T_\kappa\right)\,.
\end{equation*}

We now prove that for any $T>0$, $\mathbb{P}\left(\tau\leq T\right)\longrightarrow0$ as $\sigma$ goes to $0$. 
According to the assumption $\left\{\psi_t(x_0)\,;\,0\leq t\leq T\right\}\subset\mathcal{D}$, we know (since $\mathcal{D}$ is an open set) that there exists $\epsilon>0$ such that $\mathbb{B}\left(0;\epsilon\right)+\left\{\psi_t(x_0)\,;\,0\leq t\leq T\right\}\subset\mathcal{D}$.

Now, on the event $\left\{\tau\leq T\right\}$, we deduce that $\sup_{t\in[0;T]}\left|\left|X_t-\psi_t(x_0)\right|\right|^2>\epsilon^2$. As a consequence:
\begin{equation*}
\mathbb{P}\left(\tau\leq T\right)\leq\mathbb{P}\left(\sup_{t\in[0;T]}\left|\left|X_t-\psi_t(x_0)\right|\right|^2>\epsilon^2\right)\,,
\end{equation*}
which goes to $0$ as $\sigma$ goes to $0$, thanks to \eqref{apple}. 
This concludes the proof.
\end{proof}

\subsection{Coupling for \texorpdfstring{$t\geq T_k(\sigma)$}{t>=Tkappasigma}}\label{ss:coupling}

In \cite{Tug14c,JOTP}, Tugaut has proved the Kramers' type law for the exit-time. He has used a coupling between the diffusion of interest ($X$ here) and another diffusion that is expected to be close from $X$ if the time is sufficiently large. The main difficulty with the considered self-stabilizing diffusion is in fact that we do not have a uniform (with respect to the time) control of the law.

Here, we have proved that the nonlinear quantity appearing in the equation (that is $\frac{1}{t}\int_0^t \delta_{X_s} \rmd s$) remains stuck - with high probability - in a small ball (for $\mathbb{W}_{2k}$) of center $\delta_m$ and radius $\kappa$ for any $t\geq T_\kappa(\sigma)$. The idea is thus to replace $\frac{1}{t}\int_0^t \delta_{X_s} \rmd s$ by $\delta_m$ and to compare the new diffusion with the self-interacting one.

In other words, we consider the diffusion
\begin{equation}
\label{sandra}
Y_t=X_{T_\kappa(\sigma)}+\sigma\left(B_t-B_{T_{\kappa}(\sigma)}\right)-\int_{T_\kappa(\sigma)}^t\nabla V\left(Y_s\right)\rmd s-\int_{T_\kappa(\sigma)}^t \nabla W(Y_s-m)\rmd s\,,
\end{equation}
for any $t\geq T_\kappa(\sigma)$ and $Y_t=X_t$ if $t\leq T_\kappa(\sigma)$.

\begin{proposition}
\label{prop:coupling}
For any $\xi>0$, if $\kappa$ is small enough, we have
\begin{equation}
\label{eq:prop:coupling}
\limsup_{\sigma\to0}\mathbb{P}\left(\sup_{T_\kappa(\sigma)\leq t\leq\exp\left[\frac{2H+10}{\sigma^2}\right]}\left|X_t-Y_t\right|\geq\xi\right)\leq\sqrt{\kappa}\,.
\end{equation}
\end{proposition}

\begin{proof}
For any $t\geq T_\kappa(\sigma)$, we have
\begin{align*}
&\rmd\left|X_t-Y_t\right|^2\\
&=-2\left( X_t-Y_t\,;\,\left(\nabla V\left(X_t\right)+\nabla W\ast\mu_t(X_t) \right)-\left(\nabla V\left(Y_t\right)+\nabla W(Y_t-m)\right)\right)\rmd t\,,
\end{align*}
with the empirical measure $\mu_t:=\frac{1}{t}\int_0^t\delta_{X_s}\rmd s$. Let us define $W_m(x):=V(x)+ W(x-m)$ and $W_{\mu_t}(x):=V(x)+W\ast\mu_t(x)$. 
We thus have
\begin{align*}
\frac{\rmd}{\rmd t}\left|X_t-Y_t\right|^2=&-2\left( X_t-Y_t\,;\,\nabla W_{\mu_t}\left(X_t\right)-\nabla W_{\mu_t}\left(Y_t\right)\right)\\
&+2\left( X_t-Y_t\,;\,\nabla W(Y_t-m)-\nabla W\ast\mu_t(Y_t)\right)\,.
\end{align*}
However, $\nabla^2W_{\mu_t}=\nabla^2V+\nabla^2 W\ast\mu_t\geq\rho+\alpha>0$. So, putting $\gamma(t):=\left|X_t-Y_t\right|^2$, Cauchy-Schwarz inequality yields to
\begin{equation*}
\gamma'(t)\leq-2\left(\alpha+\rho\right)\gamma(t)+2\sqrt{\gamma(t)}\left|\nabla W(Y_t-m)-\nabla W\ast\mu_t(Y_t)\right|\,.
\end{equation*}
But by the growth assumption \eqref{domination} on $W$, we have for any probability measures $\mu,\nu$ the following control $$\left|\nabla W\ast\mu(x)-\nabla W\ast\nu(x)\right|\leq C\left(1+|x|^{2k}\right)\mathbb{W}_{2k}^{2k}\left(\mu;\nu\right)$$ where $2k$ is the degree of the polynomial $P$. We introduce the set 
\begin{equation*}
\mathcal{A}_\kappa:=\left\{\omega\in\Omega\,\,:\,\,\mathbb{W}_{2k}^{2k}\left(\mu_t;\delta_m\right)\leq\kappa^k\right\}\,.
\end{equation*}

By Markov inequality, we have $\mathbb{P}\left(\mathcal{A}_\kappa\right)\geq1-\kappa^k$. This implies for any $t\geq T_\kappa(\sigma)$ and for any $\omega\in\mathcal{A}_\kappa$:
\begin{equation*}
\gamma'(t)\leq-2\left(\alpha+\rho\right)\gamma(t)+2C\kappa^k\sqrt{\gamma(t)}\left(1+\left|Y_t\right|^{2k}\right)\,.
\end{equation*}
 However, $\gamma(t)=0$ for any $t\leq T_\kappa(\sigma)$. This means that $$\left\{t\geq0\,\,:\,\,\gamma(t)>\frac{C^2\kappa^{2k}\left(1+\left|Y_t\right|^{2k}\right)^2}{(\alpha+\rho)^2}\right\}\subset\left\{t\geq0\,\,:\,\,\gamma'(t)<0\right\}.$$ We deduce that
\begin{equation*}
\sup_{T_\kappa(\sigma)\leq t\leq\exp\left[\frac{2H+10}{\sigma^2}\right]}\gamma(t)\leq\frac{C^2\kappa^{2k}\left(1+\sup_{T_\kappa(\sigma)\leq t\leq\exp\left[\frac{2H+10}{\sigma^2}\right]}\left|Y_t\right|^{2k}\right)^2}{(\alpha+\rho)^2}\,,
\end{equation*}
if $\omega\in\mathcal{A}_\kappa$. We now consider $R>0$ such that the exit cost of the diffusion $Y$ from the ball of center $m$ and radius $R$ is at least $H+6$, meaning that $\inf \{V(x) + W(x~\!-~\!m) - V(m): x \in B(m, R)\} \geq H + 6$.

Then, by Freidlin-Wentzell theory, we deduce that\\ $\lim_{\sigma\to 0}\mathbb{P}\left(\sup_{T_\kappa(\sigma)\leq t\leq\exp\left[\frac{2H+10}{\sigma^2}\right]}\left|Y_t - m\right|\geq R\right)=0$. However, we have
\begin{eqnarray*}
&\mathbb{P}& \left(\sup_{T_\kappa(\sigma)\leq t\leq\exp\left[\frac{2H+10}{\sigma^2}\right]} \left|X_t-Y_t\right|\geq\xi\right)\\ 
& &\quad \leq \mathbb{P}\left(\sup_{T_\kappa(\sigma)\leq t\leq\exp\left[\frac{2H+10}{\sigma^2}\right]}\left|Y_t - m\right|\geq R\right)\\
& & \quad + \mathbb{P}\left(\sup_{T_\kappa(\sigma)\leq t\leq\exp\left[\frac{2H+10}{\sigma^2}\right]}\gamma(t)\geq\xi^2,\sup_{T_\kappa(\sigma)\leq t\leq\exp\left[\frac{2H+10}{\sigma^2}\right]}\left|Y_t - m\right|<R\right)\\
&&\quad+ \mathbb{P}\left(\mathcal{A}_\kappa^c\right)\,.
\end{eqnarray*}
The first term tends to $0$ as $\sigma$ goes to $0$. The second term is equal to $0$ provided that $\xi>\frac{C\kappa^{k} \big(1+ 2^{2k - 1}(R + |m|^{2k}) \big)}{\alpha+\rho}$. In other words, if $\kappa$ is small enough, the second term is equal to $0$ uniformly with respect to $\sigma$. The third term is less than $\sqrt{\kappa}$. This concludes the proof.
\end{proof}

\subsection{Proof of Theorem~\ref{dell}}\label{ss:conclusion}

Now we will prove the main Theorem~\ref{dell}. The idea of the proof is to use the fact that diffusions $Y$ and $X$ are close to each other after the deterministic stabilization time $T_{\kappa}(\sigma)$ and until some fixed deterministic time $\exp{\frac{2(H + 5)}{\sigma^2}}$, which we chose to be sufficiently big for our line of reasoning. We can control the proximity of these two diffusions by parameter $\kappa$, which represents how close time-limit of our occupation measure and $\delta_m$ are. It was already shown in Proposition~\ref{craie}, that with $\sigma \to 0$ probability of exiting before time $T_{\kappa}(\sigma)$ tends to zero, which means that we can only focus on our dynamics after the stabilization of occupation measure happens. After that, for the upper bound we show, that the event $\tau > \exp\{\frac{2(H + \delta)}{\sigma^2}\}$ is not very likely due to the fact that, for small $\sigma$, diffusion $Y$ can even leave some bigger domain in a smaller time, which contradicts the closeness of $X$ and $Y$. Same type of reasoning takes place for the lower bound. Let us now provide the rigorous proof.

Fix some $\delta, \kappa > 0$, decrease it if necessary to be $\delta < 5$. For the upper bound, consider the following inequality:
\begin{equation}\label{eq:tau_lowbound1}
    \mathbb{P}(\tau > e^{\frac{2(H + \delta)}{\sigma^2}}) \leq \mathbb{P}(\tau > e^{\frac{2(H + \delta)}{\sigma^2}}, \tau^Y_{\mathcal{D}^e} \leq e^{\frac{2(H + \delta)}{\sigma^2}}) + \mathbb{P}(\tau^Y_{\mathcal{D}^e} > e^{\frac{2(H + \delta)}{\sigma^2}}),
\end{equation}
where $\mathcal{D}^e$ is some enlargement of domain $\mathcal{D}$ such that its exit cost is equal to $H + \frac{\delta}{2}$, i.e.: 
\begin{equation*}
    \mathcal{D}^e:= \{x \in \R^d: V(x) + W(x - m) - V(m) < H + \frac{\delta}{2} \};
\end{equation*}
and $\tau^Y_{\mathcal{D}^e}$ is exit time of diffusion $Y$ from this domain, i.e.:
\begin{equation*}
    \tau^Y_{\mathcal{D}^e} := \inf \{t: Y_t \notin \mathcal{D}^e\}.
\end{equation*}
Note, that domain $\mathcal{D}^e$ (since both $V$ and $W$ are continuous and convex) satisfies the usual assumptions (see \cite{DZ}) and $d_e:= d(\mathcal{D}, \mathcal{D}^e) > 0$. By classical result of Freidlin-Wentzell theory,
\begin{equation*}
    \mathbb{P}(\tau^Y_{\mathcal{D}^e} > e^{\frac{2((H + \delta/2)  + \delta/2)}{\sigma^2}}) \xrightarrow[\sigma \to 0]{} 0.
\end{equation*}
Let us decrease $\sigma_\kappa$ if necessary, such that the quantity above will be less then $\sqrt{\kappa}$ for any $\sigma < \sigma_\kappa$. Moreover, the first probability in \eqref{eq:tau_lowbound1} can be bounded by:
\begin{equation*}
    \mathbb{P}(\tau^Y_{\mathcal{D}^e} \leq e^{\frac{2(H + \delta)}{\sigma^2}} < \tau) \leq \mathbb{P}(|X_{\tau^Y_{\mathcal{D}^e}} - Y_{\tau^Y_{\mathcal{D}^e}}| \geq d_e) \leq 2\kappa^k,
\end{equation*}
where we use Proposition \ref{prop:coupling} and decrease $\kappa$ and $\sigma_\kappa$ if necessary.

We approach the lower bound similarly and introduce the contraction of the domain $\mathcal{D}$:
\begin{equation*}
    \mathcal{D}^c := \inf \{x \in \R^d: V(x) + W(x - m) - V(m) < H - \frac{\delta}{2} \}.
\end{equation*}
If $\mathcal{D}^c$ turns out to be empty, decrease $\delta$. As previously, the domain $\mathcal{D}^c$ satisfies usual properties and has positive distance with the initial domain, that is $d_c := d(\mathcal{D}^c, \mathcal{D}) > 0$. We introduce exit-time from the contracted domain for diffusion $Y$:
\begin{equation*}
    \tau_{\mathcal{D}^c}^Y := \inf \{t: Y_t \notin \mathcal{D}^c\},
\end{equation*}
and have the following estimate:
\begin{equation*}
    \begin{aligned}
    \mathbb{P}( \tau < e^{\frac{2(H - \delta)}{\sigma^2}}) & \leq \mathbb{P}(T_{\kappa}(\sigma) < \tau < e^{\frac{2(H - \delta)}{\sigma^2}} \leq \tau_{\mathcal{D}^c}^Y) \\
    &\quad\quad + \mathbb{P}(\tau \leq T_{\kappa}(\sigma))  + \mathbb{P}(\tau_{\mathcal{D}^c}^Y \geq e^{\frac{2((H - \delta/2) - \delta/2)}{\sigma^2}}) \\
    & \leq \mathbb{P}(|X_{\tau} - Y_{\tau}| \geq d_c) ++ \mathbb{P}(\tau \leq T_{\kappa}(\sigma))+  2\kappa^k \\
    & \leq 3\kappa^k + \mathbb{P}(\tau \leq T_{\kappa}(\sigma))\leq 4\kappa^k,
    \end{aligned}
\end{equation*}
by Proposition~\ref{craie} \eqref{apple2}, with $\kappa$ and $\sigma_\kappa$ small enough. That leads to:
\begin{equation*}
    \mathbb{P}(e^{\frac{2(H - \delta)}{\sigma^2}} \leq \tau \leq e^{\frac{2(H + \delta)}{\sigma^2}}) \geq 1 - 7\kappa^k,
\end{equation*}
which proves the theorem if we consider $\kappa \to 0$, parameter that uniformly controls the convergence of $\sigma$ towards $0$.

\subsection{Proof of Corollary \ref{hphp}}
\label{s:hphp}

We can apply Theorem~\ref{dell} to the level sets of the potential $V+W\ast\delta_m$.

By definition of $\mathcal{N}$ in Corollary~\ref{hphp}, there exists a constant $\xi>0$ such that 
$\displaystyle\inf_{z\in\mathcal{N}}\left(V(z)+W(z-m)-V(m)\right)=H+3\xi$. We introduce the set 
\begin{eqnarray*}
\mathcal{K}_{H+2\xi}:=\left\{x\in\mathbb{R}^d\,\,:\,\,V(x)+W(x-m)-V(m)<H+2\xi\right\}\,.
\end{eqnarray*}
If we denote by $\tau_\xi$ the first exit time of $X$ from $\mathcal{K}_{H+2\xi}$, then we obtain
\begin{eqnarray}
\label{oz}
\lim_{\sigma \to0}\mathbb{P}\left\{\exp\left[\frac{2}{\sigma^2}\left(H+2\xi-\rho\right)\right]<\tau_\xi<\exp\left[\frac{2}{\sigma^2}\left(H+2\xi+\rho\right)\right]\right\}=1
\end{eqnarray}
for any $\rho>0$. By construction of $\mathcal{K}_{H+2\xi}$, $\mathcal{N}\subset\mathcal{K}_{H+2\xi}^c$, which implies
\begin{align*}
\mathbb{P}\left\{X_{\tau}\in\mathcal{N}\right\}\leq&\mathbb{P}\left\{X_{\tau}\notin\mathcal{K}_{H+2\xi}\right\}\\
\leq&\mathbb{P}\left\{\tau_{\xi}\leq\tau\right\}\\
\leq&\mathbb{P}\left\{\tau_{\xi}\leq\exp\left[\frac{2(H+3\xi)}{\sigma^2}\right]\right\}+\mathbb{P}\left\{\exp\left[\frac{2H+\xi}{\sigma^2}\right]\leq\tau\right\}\,.
\end{align*}
Applying \eqref{oz} with $\rho:=\xi$ to the first term and Theorem \ref{dell} to the second one, we obtain the result.

\begin{small}
\def\cprime{$'$}

\end{small}

\end{document}